\documentclass[10pt]{amsart}
\usepackage{amscd}
\usepackage{graphicx}
\usepackage[all]{xy}
\usepackage{color}

\theoremstyle{plain}
\newtheorem{theorem}{Theorem}
\newtheorem{lemma}[theorem]{Lemma}

\newtheorem{proposition}[theorem]{Proposition}

\theoremstyle{definition}

\theoremstyle{remark}
\newtheorem{remark}[theorem]{Remark}

\newtheorem*{acknowledgments}{Acknowledgments}

\newcommand\nc{\newcommand}
\nc\rnc{\renewcommand}

\nc\op{{\operatorname{op}}}
\nc\Ob{\operatorname{Ob}}
\nc\Set{\mathbf{Set}}
\nc\zzzvert {\;|\;}
\nc\zzzcolon {\colon\thinspace}
\nc\modQ {{\mathbb Q}}
\nc\modZ {{\mathbb Z}}
\nc\simeqto{\overset{\simeq}{\longrightarrow}}
\rnc\H{{\mathsf H}}
\nc\F{{\mathsf F}}
\nc\conv{\mathrm{conv}}
\nc\modF {\mathbf{F}}
\nc\modH {\mathbf{H}}
\nc\modB {\mathbf{B}}
\nc\modC {{\mathcal C}}
\nc\cv{_{\conv}}
\nc\Grp{\mathbf{Grp}}
\rnc\S{\mathfrak{S}}
\nc\Top{\mathbf{Top}}
\nc\Vect{\mathbf{Vect}}
\nc\cH{\mathcal{H}}
\nc\cB{\mathcal{B}}
\nc\bA{\mathbf{A}}
\nc\no[1]{}

\begin{document}

\title{On the category of finitely generated free groups}

\author{Kazuo Habiro}
\address{Research Institute for Mathematical Sciences\\ Kyoto
   University\\ Kyoto\\ 606-8502\\ Japan}
\email{habiro@kurims.kyoto-u.ac.jp}
\date{September 21, 2016 (First version)}

\begin{abstract}
  It is well known that the opposite $\modF ^\op$ of the category $\modF $ of
  finitely generated free groups is a Lawvere theory for groups, and
  also that $\modF $ is a free symmetric monoidal category on a
  commutative Hopf monoid, or, in other words, a PROP for commutative
  Hopf algebras.  In this paper, we give a direct, combinatorial proof
  of the latter fact, without using Lawvere theories.
\end{abstract}

\maketitle

\section{Introduction}

A {\em Lawvere theory} \cite{Lawvere} is a category $T$ with finite
products, equipped with an object $x$ such that every object of $T$ is
isomorphic to the $n$th power $x^n$ of $x$ for some $n\ge 0$.  A {\em
$T$-algebra} or an {\em algebra} over $T$ is a product preserving
functor $A\zzzcolon T\rightarrow \Set$.  It is well known that the opposite $\modF ^\op$ of
the category $\modF $ of finitely generated free groups is a Lawvere
theory, and that the $\modF ^\op$-algebras are naturally identified with
groups.  In other words, $\modF ^\op$ is a {\em Lawvere theory for
groups}.

A Lawvere theory and its opposite category are (essentially) a special
kind of {\em PROP}s.  A PROP is a symmetric strict monoidal category
$P$ equipped with one object $x$ such that $\Ob(P)=\{x^{\otimes n}\zzzvert n\ge 0\}$.
For a fixed field $k$, a {\em $P$-algebra} over $k$ is a symmetric
monoidal functor $F\zzzcolon P\rightarrow \Vect_k$, where $\Vect_k$ is the category of
$k$-vector spaces.  There are PROPs corresponding to many notions of
``algebras'' over $k$ such as (associative, unital) algebras,
coalgebras, bialgebras, Hopf algebras, etc., see \cite{Markl}.  In
this paper, we consider the PROP for {\em commutative Hopf algebras},
which is a PROP $\modH $ such that the $\modH $-algebras are identified with
commutative Hopf algebras.  The PROP $\modH $ can be defined also as the
free symmetric monoidal category generated by a commutative Hopf
monoid.  See Section \ref{sec:category-H} for the definition.

Pirashvili \cite{Pirashvili} proved that $\modF ^\op$ is a PROP for
cocommutative Hopf algebras, relying on the fact that $\modF ^\op$ is the
Lawvere theory for groups.  Equivalently, $\modF $ is a PROP for
commutative Hopf algebras, i.e., we have an isomorphism of symmetric
monoidal categories
\begin{gather}
  \label{e23}
  \modH \simeq \modF .
\end{gather}

In this paper, we will give a direct, self-contained proof of the
isomorphism \eqref{e23}, without using Lawvere theories.  Our proof is
combinatorial in the sense that it does not involve the notion of
(co)products in categories, which is used in the definition of Lawvere
theories.

An advantage of this combinatorial proof is that it admits
generalizations to categories, possibly without (co)products.  In
Section \ref{sec:motiv-from-topol}, we discuss such generalizations
with motivations in topology.  Our first motivation of writing this
paper is to provide a prototype for the proof of the results mentioned
in Sections \ref{sec:categ-handl-embedd} and
\ref{sec:categ-chord-diagr}, but we also hope that our proof of
\eqref{e23} is simpler and easier to understand.

\subsection{Organization of the paper}

The rest of this paper is organized as follows.
In Section \ref{sec:hopf-mono-symm}, we recall the notions of Hopf
monoids in symmetric monoidal categories and convolutions, and give
some basic constructions.
In Section \ref{sec:category-H}, we define the category $\modH $ and a
Hopf monoid $\H$ in $\modH $, and prove some necessary results.  For
$m,n\ge 0$, we define a surjective monoid homomorphism
\begin{gather*}
  \alpha _{m,n}\zzzcolon F_n^m \rightarrow \modH (m,n)_{\conv},
\end{gather*}
where $F_n^m$ is the the direct product of $m$ copies of the free
group $F_n=\langle x_1,\ldots,x_n\rangle $ of rank $n$, and $\modH (m,n)_{\conv}$ is
the convolution monoid on the Hom set $\modH (m,n)$.
In Section \ref{sec:category-F}, we define the symmetric monoidal
category $\modF $, a Hopf monoid $\F$ in $\modF $, and a symmetric monoidal
functor
\begin{gather*}
  T\zzzcolon \modH \longrightarrow \modF ,
\end{gather*}
which maps the Hopf monoid structure of $\H$ to that of $\F$.  We
construct a group isomorphism
\begin{gather*}
  \tau _{m,n}\zzzcolon F_n^m\simeqto\modF (m,n)_{\conv},
\end{gather*}
where $\modF (m,n)_{\conv}$ is the convolution monoid on $\modF (m,n)$.
In Section \ref{sec:isomorphism-between-}, we prove that the functor
$T\zzzcolon \modH \rightarrow \modF $ is an isomorphism by using the commutative diagram
\eqref{e24}.
In Section \ref{sec:coproducts-}, we give a direct proof of the
well-known fact that $\modH $ admits finite coproducts, without relying on
the results for $\modF $.
Section \ref{sec:motiv-from-topol} is a brief account of categories
closely related to $\modH $, which provides our motivation to study the
direct combinatorial proof of $\modH \simeq\modF $.

\begin{acknowledgments}
  The author thanks Gw\'ena\"el Massuyeau for helpful comments.

  This work is partially supported by JSPS KAKENHI Grant Number
  15K04873.
\end{acknowledgments}

\section{Hopf monoids in symmetric monoidal categories}
\label{sec:hopf-mono-symm}

Let $\modC =(\modC ,\otimes ,I,P)$ be a symmetric strict monoidal category, where
$I$ denotes the unit object and
\begin{gather*}
  P_{X,Y}\zzzcolon X\otimes Y\simeqto Y\otimes X
\end{gather*}
denotes the symmetry.

\subsection{Hopf monoids}
\label{sec:hopf-monoids}

A {\em Hopf monoid} (also called {\em Hopf algebra}) in $\modC $ is an
  object $H$ in $\modC $ equipped with morphisms
\begin{gather*}
  \mu \zzzcolon H\otimes H\rightarrow H,\quad \eta \zzzcolon I\rightarrow H,\quad \Delta \zzzcolon H\rightarrow H\otimes H,\quad \epsilon \zzzcolon H\rightarrow I,\quad
  S\zzzcolon H\rightarrow H,
\end{gather*}
called the {\em multiplication}, {\em unit}, {\em comultiplication},
  {\em counit} and {\em antipode}, respectively, satisfying
\begin{gather}
  \label{e9}
  \mu (\mu \otimes H)=\mu (H\otimes \mu ),\quad
  \mu (\eta \otimes H)=1_H=\mu (H\otimes \eta ),\\
  \label{e10}
  (\Delta \otimes H)\Delta =(H\otimes \Delta )\Delta ,\quad
  (\epsilon \otimes H)\Delta =1_H=(H\otimes \epsilon )\Delta ,\\
  \label{e11}
  \epsilon \eta =1_I,\quad
  \epsilon \mu =\epsilon \otimes \epsilon ,\quad
  \Delta \eta =\eta \otimes \eta ,\\
  \label{e18}
  \Delta \mu =(\mu \otimes \mu )(H\otimes P_{H,H}\otimes H)(\Delta \otimes \Delta ),\\
  \label{e19}
  \mu (H\otimes S)\Delta =\mu (S\otimes H)\Delta =\eta \epsilon .
\end{gather}
$H$ is said to be {\em commutative} if
\begin{gather}
  \label{e8}
  \mu P_{H,H}=\mu .
\end{gather}
As is well known, in a commutative Hopf monoid, the antipode $S$ is
involutive: $S^2=1_H$.

The iterated multiplications and comultiplications
\begin{gather*}
  \mu ^{[n]}\zzzcolon H^{\otimes n}\rightarrow H,\quad \Delta ^{[n]}\zzzcolon H\rightarrow H^{\otimes n}
\end{gather*}
for $n\ge 0$ are defined inductively by
\begin{gather*}
  \mu ^{[0]}=\eta ,\quad \mu ^{[1]}=1_H,\quad \mu ^{[n]}=\mu (\mu ^{[n-1]}\otimes H)\quad
  (n\ge 2),\\
  \Delta ^{[0]}=\epsilon ,\quad \Delta ^{[1]}=1_H,\quad
  \Delta ^{[n]}=(\Delta ^{[n-1]}\otimes H)\Delta \quad (n\ge 2).
\end{gather*}
They satisfy the generalized (co)associativity relations:
\begin{gather}
  \mu ^{[m]}(\mu ^{[k_1]}\otimes \cdots\otimes \mu ^{[k_m]})=\mu ^{[k_1+\cdots+k_m]},\\
  (\Delta ^{[k_1]}\otimes \cdots\otimes \Delta ^{[k_m]})\Delta ^{[m]}=\Delta ^{[k_1+\cdots+k_m]}
\end{gather}
for $m\ge 0,k_1,\ldots,k_m\ge 0$.

For a sequence $n_1,\ldots,n_p\ge 0$, $p\ge 0$, we set
\begin{gather*}
  \mu ^{[n_1,\ldots,n_p]} = \mu ^{[n_1]}\otimes \cdots\otimes \mu ^{[n_p]},\\
  \Delta ^{[n_1,\ldots,n_p]} = \Delta ^{[n_1]}\otimes \cdots\otimes \Delta ^{[n_p]}.
\end{gather*}

\subsection{Convolutions}

A {\em monoid} in $\modC $ is an object $H$ equipped with morphisms
$\mu \zzzcolon H\otimes H\rightarrow H$ and $\eta \zzzcolon I\rightarrow H$ satisfying \eqref{e9}, and a {\em
comonoid} in $\modC $ is an object $H$ equipped with morphisms
$\Delta \zzzcolon H\rightarrow H\otimes H$ and $\epsilon \zzzcolon H\rightarrow I$ satisfying \eqref{e10}.

Let $A=(A,\mu _A,\eta _A)$ be a monoid in $\modC $, and $C=(C,\Delta _C,\epsilon _C)$ a
comonoid in $\modC $.  Then the set $\modC (C,A)$ of morphisms from $C$ to $A$
in $\modC $ is equipped with a monoid structure, with multiplication given
by the {\em convolution}
\begin{gather*}
  f * f' := \mu _A(f\otimes f')\Delta _C
\end{gather*}
for $f,f'\in \modC (C,A)$, and with unit given by $\eta _A\epsilon _C\in \modC (C,A)$.  Let
$\modC (C,A)_{\conv}$ denote this monoid.

In this paper, we use convolutions in the following special situation.

Let $H$ be a Hopf monoid in $\modC $.  Then, for $m\ge 0$, the tensor power
$H^{\otimes m}$ has a monoid structure
\begin{gather*}
  \mu _m\zzzcolon H^{\otimes m}\otimes H^{\otimes m}\rightarrow H^{\otimes m},\quad
  \eta _m:=\eta ^{\otimes m}\zzzcolon I\rightarrow H^{\otimes m},
\end{gather*}
where $\mu _m$ is defined inductively by $\mu _0=1_I$ and
\begin{gather*}
  \mu _{m+1}=(\mu _m\otimes \mu )(H^{\otimes m}\otimes P_{H,H^{\otimes m}}\otimes H)\quad (m\ge 0).
\end{gather*}
Similarly, $H^{\otimes m}$ has a comonoid structure
\begin{gather*}
  \Delta _m\zzzcolon H^{\otimes m}\rightarrow H^{\otimes m}\otimes H^{\otimes m},\quad
  \epsilon _m:=\epsilon ^{\otimes m}\zzzcolon H^{\otimes m}\rightarrow I,
\end{gather*}
where $\Delta _m$ is defined inductively by $\Delta _0=1_I$ and
\begin{gather*}
  \Delta _{m+1}=(H^{\otimes m}\otimes P_{H^{\otimes m},H}\otimes H)(\Delta _m\otimes \Delta )\quad (m\ge 0).
\end{gather*}
(The monoid $(H^{\otimes m},\mu _m,\eta _m)$ and the comonoid
$(H^{\otimes m},\Delta _m,\eta _m)$ are part of the Hopf monoid structure on
$H^{\otimes m}$, with the antipode $S_m:=S^{\otimes m}\zzzcolon H^{\otimes m}\rightarrow H^{\otimes m}$.)

For $f,g\in H^{\otimes m}\rightarrow H^{\otimes n}$, $m,n\ge 0$, the {\em convolution}
$f*g\zzzcolon H^{\otimes m}\rightarrow H^{\otimes n}$ of $f$ and $g$ is given by
\begin{gather}
  \label{e20}
  f*g=\mu _n(f\otimes g)\Delta _m.
\end{gather}
The operation $*$ is associative and unital with unit $\eta _n\epsilon _m$.
Hence $*$ gives the set $\modC (H^{\otimes m},H^{\otimes n})$ a monoid structure.  Let
$\modC (H^{\otimes m},H^{\otimes n})\cv$ denote this monoid.

\subsection{The operation $\vee$}
\label{sec:operation-ts}

We also need the following variant of convolution.  For
$f\zzzcolon H^{\otimes m}\rightarrow H^{\otimes n}$, $g\zzzcolon H^{\otimes m'}\rightarrow H^{\otimes n}$, $m,m',n\ge 0$, set
\begin{gather*}
  f\vee f' := \mu _n(f\otimes f')\zzzcolon H^{\otimes m+m'}\rightarrow H^{\otimes n}
\end{gather*}
It is easy to see that $\vee$ is associative and unital with unit
$\eta _n\zzzcolon I\rightarrow H^{\otimes n}$.

\begin{lemma}
  \label{r7}
  Suppose that $H$ is commutative.  Then we have
  \begin{gather}
    \label{e12}
    (f\vee f')*(g\vee g')=(f*g)\vee(f'*g')
  \end{gather}
  for $f,g\zzzcolon H^{\otimes m}\rightarrow H^{\otimes n}$ and $f',g'\zzzcolon H^{\otimes m'}\rightarrow H^{\otimes n}$,
  $m,m',n\ge 0$.
\end{lemma}

\begin{proof}
  We obtain \eqref{e12} as follows.
  \begin{gather*}
    \begin{split}
      &(f\vee f')*(g\vee g')\\
      =& \mu _n(\mu _n(f\otimes f') \otimes  \mu _n(g\otimes g'))\Delta _{m+m'}\\
      =& \mu _n(\mu _n\otimes \mu _n)(f\otimes f'\otimes g\otimes g')(H^{\otimes m}\otimes P_{H^{\otimes m},H^{\otimes m'}}\otimes H^{\otimes m'})(\Delta _{m}\otimes \Delta _{m'})\\
      =& \mu _n^{[4]}(H^{\otimes n}\otimes P_{H^{\otimes n},H^{\otimes n}}\otimes H^{\otimes n})(f\otimes g\otimes f'\otimes g')(\Delta _{m}\otimes \Delta _{m'}),
    \end{split}
  \end{gather*}
  where for $k\ge 0$, $\mu _n^{[k]}\zzzcolon (H^{\otimes n})^{\otimes k}\rightarrow H^{\otimes n}$ is the
  $k$-fold iterated multiplication for the monoid $H^{\otimes n}$.
  Commutativity of $H$ implies commutativity of $H^{\otimes n}$:
  $\mu _nP_{H^{\otimes n},H^{\otimes n}}=\mu _n$, and hence
  $\mu _n^{[4]}(H^{\otimes n}\otimes P_{H^{\otimes n},H^{\otimes n}}\otimes H^{\otimes n})=\mu _n^{[4]}$.
  Therefore we have
    \begin{gather*}
    \begin{split}
      (f\vee f')*(g\vee g')
      =& \mu _n^{[4]}(f\otimes g\otimes f'\otimes g')(\Delta _{m}\otimes \Delta _{m'})\\
      =& \mu _n(\mu _n\otimes \mu _n)(f\otimes g\otimes f'\otimes g')(\Delta _{m}\otimes \Delta _{m'})\\
      =& \mu _n\bigl((\mu _n(f\otimes g)\Delta _m)\otimes (\mu _n(f'\otimes g')\Delta _m)\bigr)\\
      =& (f* g)\vee (f'*g').
    \end{split}
  \end{gather*}
\end{proof}

By Lemma \ref{r7}, when $H$ is commutative, $\vee$ gives rise to a
monoid homomorphism
\begin{gather}
  \label{e29}
  \vee\zzzcolon \modC (H^{\otimes m},H^{\otimes n})\cv\times \modC (H^{\otimes m'},H^{\otimes n})\cv\rightarrow \modC (H^{\otimes m+m'},H^{\otimes n})\cv.
\end{gather}

\subsection{Permutation morphisms}

For $n\ge 0$, let $\S_n$ denote the symmetric group of order $n$.
Define a homomorphism
\begin{gather*}
  \S_n\rightarrow \modH (n,n),\quad \sigma \mapsto P_\sigma =P^H_\sigma 
\end{gather*}
by
\begin{gather*}
  P_{(i,i+1)}=H^{\otimes i-1}\otimes P_{H,H}\otimes H^{\otimes n-i-1}
\end{gather*}
for $i=1,\ldots ,n-1$.

\subsection{Generalized $\Delta \mu $-relation}

We have the following generalization of \eqref{e11},
\eqref{e18}:
\begin{gather}
  \label{e22}
  \Delta ^{[n]}\mu ^{[m]} = (\mu ^{[m]})^{\otimes n}P_{t_{m,n}}(\Delta ^{[n]})^{\otimes m}
\end{gather}
for $m,n\ge 0$, where $t_{m,n}\in \S_{mn}$ is defined by
\begin{gather}
  \label{e17}
  t_{m,n}((l-1)n+k)=(k-1)m+l
\end{gather}
for $1\le k\le n$, $1\le l\le m$.

\section{The category $\modH $}
\label{sec:category-H}

\subsection{Definition of $\modH $}
Let $\modH $ denote the free symmetric strict monoidal category generated
by a commutative Hopf monoid $\H$.  The objects in $\modH $ are tensor
powers $\H^{\otimes n}$, $n\ge 0$, which are usually identified with $n$.  The
morphisms in $\modH $ are obtained by taking tensor products and
compositions of copies of the morphisms
\begin{gather*}
  P_{m,n}=P_{\H^{\otimes m},\H^{\otimes n}}\zzzcolon m+n\rightarrow n+m\quad (m,n\ge 0),\\
  \mu =\mu _\H\zzzcolon 2\rightarrow 1,\quad
  \eta =\eta _\H\zzzcolon 0\rightarrow 1,\quad
  \Delta =\Delta _\H\zzzcolon 1\rightarrow 2,\quad
  \epsilon =\epsilon _\H\zzzcolon 1\rightarrow 0,\\
  S=S_\H\zzzcolon 1\rightarrow 1,
\end{gather*}
satisfying the axioms of symmetric monoidal category and commutative
Hopf monoid, and satisfying no relations that are not implied by these
axioms.

{\em Freeness} or {\em universality} of $\modH $ is the fact that for any
symmetric strict monoidal category $\modC $ and any commutative Hopf
monoid $A=(A,\mu _A,\eta _A,\Delta _A,\epsilon _A,S_A)$ in $\modC $, there is a unique
strict symmetric monoidal functor $F\zzzcolon \modH \rightarrow \modC $ mapping the Hopf monoid
$\H$ to the Hopf monoid $A$, i.e., $F(\H)=A$ and
$F(\mu _\H)=\mu _A,\ldots,F(S_\H)=S_A$.  In other words, the category $\modH $
is a PROP for commutative Hopf algebras.

\subsection{Factorization of $\modH $}

\begin{lemma}
  \label{factorization}
  Every morphism $f\zzzcolon m\rightarrow n$ in $\modH $ admits factorization
  \begin{gather}
    \label{e16}
    f= \mu ^{[q_1,\ldots,q_n]}P_\sigma (S^{e_1}\otimes \cdots\otimes S^{e_s})\Delta ^{[p_1,\ldots,p_m]},
  \end{gather}
  where $s,p_1,\ldots,p_m,q_1,\ldots,q_n\ge 0$ with
  $s=p_1+\cdots+p_m=q_1+\cdots+q_n$, $e_1,\ldots,e_s\in \{0,1\}$ and
  $\sigma \in \S_s$.
\end{lemma}

\begin{proof}
  Let $\modH ^0$ (resp. $\modH ^+$, $\modH ^-$) denote the monoidal subcategory of
  $\modH $ generated by the object $\H=1$ and morphisms $\{P_{1,1},S\}$
  (resp., $\{\mu ,\eta \}$, $\{\Delta ,\epsilon \}$).  We also use the symbols $\modH ^*$
  ($*=0,+,-$) for categories to denote the set of morphisms,
  $\coprod_{m,n\ge 0}\modH ^*(m,n)$.  We can easily verify the following.
\begin{gather}
  \label{e30}
  \modH ^0=\{P_\sigma (S^{e_1}\otimes \cdots\otimes S^{e_s})\zzzvert s\ge 0,\;\sigma \in \S_s,\;(e_1,\ldots,e_s)\in \{0,1\}^s\},\\
  \label{e32}
  \modH ^+=\{\mu ^{[q_1,\ldots,q_n]}\zzzvert n\ge 0,\;q_1,\ldots,q_n\ge 0\},\\
  \label{e34}
  \modH ^-=\{\Delta ^{[p_1,\ldots,p_m]}\zzzvert m\ge 0,\;p_1,\ldots,p_m\ge 0\}.
\end{gather}

We have the following inclusions
\begin{gather}
  \label{e36}
  \modH ^-\modH ^+\subset \modH ^+\modH ^0\modH ^-,\\
  \label{e35}
  \modH ^0\modH ^+\subset \modH ^+\modH ^0,\quad \modH ^-\modH ^0\subset \modH ^0\modH ^-.
\end{gather}
\eqref{e36} follows from the generalized $\Delta \mu $-relation \eqref{e22}.
To prove the first inclusion in \eqref{e35}, it suffices to show that
$\modH ^0g\subset \modH ^+\modH ^0$ for any $g=1_k\otimes f\otimes 1_l$ with $f\in \{\mu ,\eta \}$ and
$k,l\ge 0$, since the category $\modH ^+$ is generated by such $g$'s.  The
case $f=\mu $ (resp. $f=\eta $) can be checked by using naturality of the
symmetry and $S\mu =\mu P_{1,1}(S\otimes S)$ (resp. $S\eta =\eta $).  The second
inclusion in \eqref{e35} can be proved similarly.

Using \eqref{e36} and \eqref{e35}, we see that the set $\modH ^+\modH ^0\modH ^-$
is closed under composition,
i.e. $\modH ^+\modH ^0\modH ^-\modH ^+\modH ^0\modH ^-\subset \modH ^+\modH ^0\modH ^-$.  Since the category
$\modH $ is generated by $\modH ^+\modH ^0\modH ^-$, we have $\modH =\modH ^+\modH ^0\modH ^-$.  Then
the lemma follows from \eqref{e30}--\eqref{e34}.
\end{proof}

\begin{remark}
  \label{r13}
  Lemma \ref{factorization} admits several natural generalizations,
  which can be proved by similar arguments.
  \begin{itemize}
  \item We do not need the freeness of the category $\modH $: Lemma
  \ref{factorization} generalizes to any symmetric monoidal category
  generated by a commutative Hopf monoid possibly with additional
  relations.
  \item One can generalize Lemma \ref{factorization} to a symmetric
  strict monoidal category generated by a Hopf monoid, which is not
  assumed to be commutative.  In this case, $e_1,\ldots,e_s$ should be
  taken as arbitrary nonnegative integers.  If we assume that the
  antipode is invertible, then $e_1,\ldots,e_s$ should be taken as
  arbitrary integers.
\item One can generalize Lemma \ref{factorization} to braided monoidal
  category generated by a Hopf monoid.  In this case the permutation
  morphism $P_\sigma $ should be replaced with a braid.
\end{itemize}
\end{remark}

\subsection{The homomorphism $\alpha _{m,n}\zzzcolon F_n^m\rightarrow \modH (m,n)_{\conv}$}

For $n\ge 0$, define a monoid homomorphism
\begin{gather*}
  \alpha _{1,n}\zzzcolon F_n\rightarrow \modH (1,n)_{\conv}
\end{gather*}
by
\begin{align*}
  \alpha _{1,n}\left(x_j^{(-1)^e}\right) &= \eta _{j-1}\otimes S^e\otimes \eta _{n-j}
\end{align*}
for $j=1,\ldots,n$ and $e=0,1$.  This is well defined since
$\eta _{j-1}\otimes 1_1\otimes \eta _{n-j}$ and $\eta _{j-1}\otimes S\otimes \eta _{n-j}$ are
convolution-inverse to each other.

Now we extend the definition of $\alpha _{1,n}$ as follows.  For $m,n\ge 0$,
define a map
\begin{gather*}
  \alpha _{m,n}\zzzcolon F_n^m\rightarrow \modH (m,n)\cv
\end{gather*}
by
\begin{gather}
  \label{e27}
  \alpha _{m,n}(w_1,\ldots,w_m) = \alpha _{1,n}(w_1)\vee\cdots\vee \alpha _{1,n}(w_m)
\end{gather}
for $w_1,\ldots,w_m\in F_n$, where the $\vee$ are defined in Section
\ref{sec:operation-ts}.  Clearly, this definition of $\alpha _{m,n}$ is
compatible with that of $\alpha _{1,n}$ above.  If $m=0$, then the right
hand side of \eqref{e27} should be understood as the unit in $\vee$,
which is $\eta _n$.

\begin{lemma}
  \label{r2}
  The map $\alpha _{m,n}$ is a monoid homomorphism.
\end{lemma}

\begin{proof}
  We have
  \begin{gather*}
    \begin{split}
      &\alpha _{m,n}(w_1,\ldots,w_m)*\alpha _{m,n}(w'_1,\ldots,w'_m)\\
      =&(\alpha _{1,n}(w_1)\vee\cdots\vee \alpha _{1,n}(w_m))*(\alpha _{1,n}(w'_1)\vee\cdots\vee\alpha _{1,n}(w'_m))\\
      =&(\alpha _{1,n}(w_1)*\alpha _{1,n}(w'_1))\vee\cdots\vee(\alpha _{1,n}(w_m)*\alpha _{1,n}(w'_m))\quad \text{(by Lemma \ref{r7})}\\
      =&\alpha _{1,n}(w_1 w'_1)\vee\cdots\vee \alpha _{1,n}(w_m w'_m)\\
      =&\alpha _{m,n}(w_1 w'_1,\ldots,w_m w'_m)
    \end{split}
  \end{gather*}
  and
  \begin{gather*}
    \alpha _{m,n}(1,\ldots,1)
    =\alpha _{1,n}(1)\vee\cdots\vee \alpha _{1,n}(1)
    =\eta _n\epsilon \vee\cdots\vee\eta _n\epsilon 
    =\eta _n\epsilon _m.
  \end{gather*}
\end{proof}

\subsection{Surjectivity of $\alpha _{m,n}$}

For $i\in \{1,\ldots,m\}$, $j\in \{1,\ldots,n\}$ and $e\in \{0,1\}$, set
\begin{gather*}
  y_{i,j,e}:=(1,\ldots,{x_j^{(-1)^e}},\ldots,1)\in F_n^m
\end{gather*}
with $x_j^{(-1)^e}$ in the $i$th place.  As a monoid, $F_n^m$ is
generated by the $y_{i,j,e}$.  As a group, $F_n^m$ is generated by the
$y_{i,j,0}$.  Set
\begin{gather*}
  \begin{split}
    y^\H_{i,j,e}:=&\alpha _{m,n}(y_{i,j,e})
    =(\eta _{j-1}\otimes 1_1\otimes \eta _{n-j})S^e(\epsilon _{i-1}\otimes 1_1\otimes \epsilon _{m-i}).
  \end{split}
\end{gather*}
Since $y_{i,j,1}=y_{i,j,0}^{-1}$, it follows that $y^\H_{i,j,0}$ and
$y^\H_{i,j,1}$ are convolution-inverse to each other.

\begin{lemma}
  \label{r6}
  The homomorphism $\alpha _{m,n}$ is surjective.  (Thus, $\modH (m,n)\cv$ is a
  group, since it is a homomorphic image of the group $F_n^m$.)
\end{lemma}

\begin{proof}
  Define the {\em size} $s(f)$ of a morphism $f\zzzcolon m\rightarrow n$ in $\modH $ to be
  the least integer $s\ge 0$ such that $f$ is expressed as in
  \eqref{e16}.  We have $s(f)=0$ if and only if $f=\eta _n\epsilon _m$, the
  convolution unit in $\modH (m,n)$.

  To prove $\alpha _{m,n}$ surjective, it suffices to prove the following
  claim: If $s(f)>0$, then there is an $f'\zzzcolon m\rightarrow n$ with $s(f')<s(f)$
  such that
  \begin{gather}
    \label{e21}
    f=y^\H_{i,j,e}*f'
  \end{gather}
  for some $i,j,e$.

  Let us prove this claim.  We assume the situation in Lemma
  \ref{factorization}.  Set
  \begin{gather*}
    i:=\min\{i'\in \{1,\ldots,m\}\zzzvert p_{i'}>0\},
  \end{gather*}
  which is well-defined since $s=s(f)>0$.  Note that
  $p_1=\cdots=p_{i-1}=0$ and $p_i>0$.  Set $e:=e_1$.  Let
  $j\in \{1,\ldots,n\}$ be the unique integer such that
  \begin{gather*}
    q_1+\cdots+q_{j-1}+1\le \sigma (1)\le q_1+\cdots+q_j.
  \end{gather*}
  We have $q_j\ge 1$.  We may assume $\sigma (1)=q_1+\cdots+q_{j-1}+1$, since
  if not we can modify $\sigma $ in \eqref{e16} using commutativity of $\H$
  so that we have $\sigma (1)=q_1+\cdots+q_{j-1}+1$.  Then one can check
  \eqref{e21} with
  \begin{gather*}
    f'=\mu ^{[q_1,\ldots,q_{j-1},q_j-1,q_{j+1},\ldots,q_n]}P_{\sigma '}(S^{e_2}\otimes \cdots\otimes S^{e_s})\Delta ^{[0,\ldots,0,p_i-1,p_{i+1},\ldots,p_m]},
  \end{gather*}
  where
  $\sigma '\in \S_{s-1}$ is the composition of
  \begin{gather*}
    \{1,\ldots,s-1\}\underset{\simeq}{\longrightarrow}
    \{2,\ldots,s\}\overset{\sigma \zzzvert _{\{2,\ldots,s\}}}{\underset{\simeq}{\longrightarrow}}
    \{1,\ldots,s\}\setminus \{\sigma (1)\}\underset{\simeq}{\longrightarrow}
    \{1,\ldots,s-1\},
  \end{gather*}
  where the unnamed arrows are the unique order-preserving bijections.
  This completes the proof.
\end{proof}

\section{The category $\modF $}
\label{sec:category-F}

\subsection{The category $\modF $ of finitely generated free groups}

Let $\modF $ denote the full subcategory of the category of groups, such
that $\Ob(\modF )=\{F_n\zzzvert n\ge 0\}$.  We usually identify the object $F_n$
with the integer $n$.

Define a bijection
\begin{gather*}
  \tau _{m,n}\zzzcolon F_n^m \simeqto \modF (m,n)
\end{gather*}
by
\begin{gather}
  \label{e3}
  \tau _{m,n}(w_1,\ldots,w_m)(x_i)=w_i
\end{gather}
for $w_1,\ldots,w_m\in F_n$ and $i=1,\ldots,m$.  We also use the
notation
\begin{gather*}
  [w_1,\ldots,w_m]_{m,n}=\tau _{m,n}(w_1,\ldots,w_m).
\end{gather*}

The composition rule for this bracket notation is as follows.
\begin{gather}
  \label{e2}
  [w_1,\ldots,w_m]_{m,n}[v_1,\ldots,v_l]_{l,m}
  =[v_1(w_1,\ldots,w_m),\ldots,v_l(w_1,\ldots,w_m)]_{l,n},
\end{gather}
where $v_i(w_1,\ldots,v_m)\in F_n$ is obtained from $v_i$ by
substituting $w_j$ for $x_j$ for $j=1,\ldots,m$.

\subsection{Symmetric monoidal structure of $\modF $}

The category $\modF $ has a monoidal structure with tensor functor $\otimes $
given by the free product, and the monoidal unit given by $0=F_0$.  We
have $m\otimes n=m+n$ for all $m,n\ge 0$.  The tensor product
$f\otimes f'\zzzcolon m+m'\rightarrow n+n'$ of $f\zzzcolon m\rightarrow n$ and $f'\zzzcolon m'\rightarrow n'$ is the unique
homomorphism $f\otimes f'\zzzcolon F_{m+m'}\rightarrow F_{n+n'}$ such that
\begin{gather*}
  (f\otimes f')(x_i)=
  \begin{cases}
    f(x_i)&\text{for $i=1,\ldots,m$},\\
    s_{n}(f'(x_{i-m}))&\text{for $i=m+1,\ldots,m+m'$}.
  \end{cases}
\end{gather*}
Here the homomorphism $s_n\zzzcolon F_{n'}\rightarrow F_{n+n'}$ is defined by
$s_n(x_i)=x_{i+n}$.

In the bracket notation, we have
\begin{gather}
  \label{e1}
  [w_1,\ldots,w_m]_{m,n}\otimes [w'_1,\ldots,w'_{m'}]_{m',n'}
  =[w_1,\ldots,w_m,s_n(w'_1),\ldots,s_n(w'_{m'})]_{m+n,m'+n'}.
\end{gather}

The symmetry in $\modF $ are defined by
\begin{gather*}
  P_{m,n} = [x_{m+1},\ldots,x_{m+n},x_1,\ldots,x_m]_{m+n,n+m}\zzzcolon m+n\rightarrow n+m.
\end{gather*}

\subsection{Commutative Hopf monoid in $\modF $}

The following is well known.

\begin{proposition}
  \label{r44}
  The object $1$ in $\modF $ has a commutative Hopf monoid structure with
  the multiplication, unit, comultiplication, counit and antipode
  defined by
  \begin{gather*}
    \mu _\F=[x_1,x_1]_{2,1},\quad
    \eta _\F=[1]_{0,1},\quad
    \Delta _\F=[x_1x_2]_{1,2},\quad
    \epsilon _\F=[\;]_{1,0},\quad
    S_\F=[x_1^{-1}]_{1,1}.
  \end{gather*}
\end{proposition}

Let $\F=(1,\mu _\F,\eta _\F,\Delta _\F,\epsilon _\F,S_\F)$ denote this Hopf monoid in
$\modF $.  We often omit the subscript $\F$ from the notation.

\begin{proof}
  The proposition can be checked by easy computations.  For
  example, $\Delta \mu =(\mu \otimes \mu )(1\otimes P_{1,1}\otimes 1)(\Delta \otimes \Delta )$ can be checked by
  \begin{gather*}
    \Delta \mu 
    =[x_1x_2]_{1,2}[x_1,x_1]_{2,1}
    =[x_1x_2,x_1x_2]_{2,2},\\
    \begin{split}
      &(\mu \otimes \mu )(1\otimes P_{1,1}\otimes 1)(\Delta \otimes \Delta )
      =[x_1,x_1,x_2,x_2]_{4,2}[x_1,x_3,x_2,x_4]_{4,4}[x_1x_2,x_3x_4]_{2,4}\\
      &\quad \quad =[x_1,x_2,x_1,x_2]_{4,2}[x_1x_2,x_3x_4]_{2,4}
    =[x_1x_2,x_1x_2]_{2,2}.
    \end{split}
  \end{gather*}
   \no{\marginpar{\Tiny Here is a proof.}
   The proof is by the following computations.

   Associativity:
   \begin{gather*}
     \mu (\mu \otimes 1)=[x_1,x_1]_{2,1}[x_1,x_1,x_2]_{3,2}=[x_1,x_1,x_1]_{3,1},\\
     \mu (1\otimes \mu )=[x_1,x_1]_{2,1}[x_1,x_2,x_2]_{3,2}=[x_1,x_1,x_1]_{3,1}.
   \end{gather*}
   Unitality:
   \begin{gather*}
     \mu (\eta \otimes 1)=[x_1,x_1]_{2,1}[x_2]_{1,2}=x_1,\\
     \mu (1\otimes \eta )=[x_1,x_1]_{2,1}[x_1]_{1,2}=x_1.
   \end{gather*}
   Coassociativity:
   \begin{gather*}
     (\Delta \otimes 1)\Delta =[x_1x_2,x_3]_{2,3}[x_1x_2]_{1,2}=[x_1x_2x_3]_{1,3},\\
     (1\otimes \Delta )\Delta =[x_1,x_2x_3]_{2,3}[x_1x_2]_{1,2}=[x_1x_2x_3]_{1,3}.
   \end{gather*}
   Counitality:
   \begin{gather*}
     (\epsilon \otimes 1)\Delta =[1,x_1]_{2,1}[x_1x_2]_{1,2}=1x_1=x_1,\\
     (1\otimes \epsilon )\Delta =[x_1,1]_{2,1}[x_1x_2]_{1,2}=x_11=x_1.
   \end{gather*}
   $\epsilon \eta =1_0$:
   \begin{gather*}
     \epsilon \eta =[1]_{1,0}[\;]_{0,1}=[\;]_{0,0}=1_0.
   \end{gather*}
   $\Delta \eta =\eta \otimes \eta $:
   \begin{gather*}
     \Delta \eta =[x_1x_2]_{1,2}[\;]_{0,1}=[\;]_{0,2}=\eta \otimes \eta .
   \end{gather*}
   $\epsilon \mu =\epsilon \otimes \epsilon $:
   \begin{gather*}
     \epsilon \mu =[1]_{1,0}[x_1,x_1]_{2,1}=[1,1]_{2,0}=\epsilon \otimes \epsilon .
   \end{gather*}
   $\Delta \mu =(\mu \otimes \mu )(1\otimes P_{1,1}\otimes 1)(\Delta \otimes \Delta )$:
   \begin{gather*}
     \Delta \mu 
     =[x_1x_2]_{1,2}[x_1,x_1]_{2,1}
     =[x_1x_2,x_1x_2]_{2,2},\\
     \begin{split}
     &(\mu \otimes \mu )(1\otimes P_{1,1}\otimes 1)(\Delta \otimes \Delta )
       =[x_1,x_1,x_2,x_2]_{4,2}[x_1,x_3,x_2,x_4]_{4,4}[x_1x_2,x_3x_4]_{2,4}\\
       &\quad \quad =[x_1,x_2,x_1,x_2]_{4,2}[x_1x_2,x_3x_4]_{2,4}
     =[x_1x_2,x_1x_2]_{2,2}.
     \end{split}
   \end{gather*}
   Antipode:
   \begin{gather*}
     \mu (1\otimes S)\Delta 
     =[x_1,x_1]_{2,1}[x_1,x_2^{-1}][x_1x_2]_{1,2}
     =[x_1,x_1]_{2,1}[x_1x_2^{-1}]_{1,2}
     =[x_1x_1^{-1}]_{1,2}
     =[1]_{1,2},\\
     \mu (S\otimes 1)\Delta 
     =[x_1,x_1]_{2,1}[x_1^{-1},x_2][x_1x_2]_{1,2}
     =[x_1,x_1]_{2,1}[x_1^{-1}x_2]_{1,2}
     =[x_1^{-1}x_1]_{1,2}
     =[1]_{1,2},\\
     \eta \epsilon =[\;]_{0,1}[1]_{1,0}=[1]_{1,2}.
   \end{gather*}
   Commutativity:
   \begin{gather*}
     \mu P_{1,1}=[x_1,x_1]_{2,1}[x_2,x_1]_{2,2}=[x_1,x_1]_{2,1}=\mu .
   \end{gather*}
   }
\end{proof}

By Proposition \ref{r44} and universality of $\modH $, there is a unique
symmetric monoidal functor
\begin{gather}
  T\zzzcolon \modH \longrightarrow\modF ,
\end{gather}
which maps the Hopf monoid $\H$ in $\modH $ to the Hopf monoid $\F$ in $\modF $.

\subsection{The group $\modF (m,n)\cv$}

\begin{proposition}
  \label{r49}
  For $m,n\ge 0$, the bijection $\tau _{m,n}$ gives rise to a group
  isomorphism
  \begin{gather*}
    \tau _{m,n}\zzzcolon F_n^m\simeqto \modF (m,n)\cv.
  \end{gather*}
\end{proposition}

\begin{proof}
It suffices to see that $\tau _{m,n}$ is a monoid homomorphism, i.e., we
have
\begin{gather}
  \label{e4}
	[w_1,\ldots,w_m]_{m,n}*[w'_1,\ldots,w'_m]_{m,n}
	=
	[w_1w'_1,\ldots,w_m w'_m]_{m,n}.
\end{gather}
  for $w_1,\ldots,w_m,w'_1,\ldots,w'_m\in F_n$, and
  \begin{gather}
    \label{e5}
    \eta _n\epsilon _m=[1,\ldots,1]_{m,n}.
  \end{gather}
  \eqref{e4} can be checked as follows:
  \begin{gather*}
    \begin{split}
      &[w_1,\ldots,w_m]_{m,n}*[w'_1,\ldots,w'_m]_{m,n}\\
      =&\mu _n([w_1,\ldots,w_m]_{m,n}\otimes [w'_1,\ldots,w'_m]_{m,n})\Delta _n\\
      =&[x_1,\ldots,x_n,x_1,\ldots,x_n]_{2n,n}
      [w_1,\ldots,w_m,s_n(w'_1),\ldots,s_n(w'_m)]_{2m,2n}\\
      &\quad
      \times [x_1x_{m+1},\ldots,x_m x_{2m}]_{m,2m}\\
      =&[w_1,\ldots,w_m,w'_1,\ldots,w'_m]_{2m,n}[x_1x_{m+1},\ldots,x_m x_{2m}]_{m,2m}\\
      =&[w_1w'_1,\ldots,w_m w'_m]_{m,n}.
    \end{split}
  \end{gather*}
  \eqref{e5} can be checked easily.
\end{proof}

\section{Isomorphism between $\modH $ and $\modF $}
\label{sec:isomorphism-between-}

\begin{theorem}
  \label{H=F}
  The functor $T\zzzcolon \modH \longrightarrow\modF $ is an isomorphism of symmetric
  strict monoidal categories.
\end{theorem}

As mentioned in the introduction, the above result is well known
\cite{Pirashvili}.  The proof below is a direct, combinatorial proof,
not relying on Lawvere theories.

\begin{proof}
  It suffices to prove that $T_{m,n}\zzzcolon \modH (m,n)\rightarrow \modF (m,n)$ is bijective
  for $m,n\ge 0$.

  Consider the following diagram of group homomorphisms
  \begin{gather}
    \label{e24}
    \vcenter{\xymatrix{
      \modH (m,n)\cv
      \ar[rr]^{T_{m,n}}
      &
      &
      \modF (m,n)\cv
      \\
      &
      F_n^m
      \ar@{->>}[lu]^{\alpha _{m,n}}
      \ar[ru]_{\tau _{m,n}}^{\simeq}
      &
    }}
  \end{gather}
  Here, $T_{m,n}$ is a homomorphism since $T$ is a monoidal functor.
  Since $\tau _{m,n}$ is an isomorphism (Proposition \ref{r49}) and
  $\alpha _{m,n}$ is surjective (Lemma \ref{r6}), it suffices to prove that
  the diagram commutes.

  Since the group $F_n^m$ is generated by $y_{i,j,0}$ for
  $i\in \{1,\ldots,m\}$ and $j\in \{1,\ldots,n\}$, we have only to check
  \begin{gather}
    \label{e14}
    T_{m,n}\alpha _{m,n}(y_{i,j,0})= \tau _{m,n}(y_{i,j,0}).
  \end{gather}
  Indeed,
  \begin{gather*}
    \begin{split}
      T_{m,n}\alpha _{m,n}(y_{i,j,0})
      =&T_{m,n}(y^\H_{i,j,0})\\
      =&T_{m,n}\left((\eta _\H^{\otimes j-1}\otimes 1_1\otimes \eta _\H^{\otimes n-j})(\epsilon _\H^{\otimes i-1}\otimes 1_1\otimes \epsilon _\H^{\otimes m-i})\right)\\
      =&(\eta _\F^{\otimes j-1}\otimes 1_1\otimes \eta _\F^{\otimes n-j})(\epsilon _\F^{\otimes i-1}\otimes 1_1\otimes \epsilon _\F^{\otimes m-i})\\
      =&[x_j]_{1,n}[\overset{1}1,\ldots,\overset{i}{x_1\negthinspace\negthinspace\phantom{1}},\ldots,\overset{m}1]_{m,1}\\
      =&[\overset{1}1,\ldots,\overset{i}{x_j\negthinspace\negthinspace\phantom{1}},\ldots,\overset{m}1]_{m,n}\\
      =&y_{i,j,0}.
    \end{split}
  \end{gather*}
\end{proof}

\section{Coproducts in $\modH $}
\label{sec:coproducts-}

As is well known, $\modF $ admits finite coproducts given by free
products.  Since $\modH \simeq\modF $, the category $\modH $ admits finite
coproducts.  Here we provide a direct proof of this fact, without
relying on the corresponding result for $\modF $ in Sections
\ref{sec:category-F} and \ref{sec:isomorphism-between-}.

\begin{lemma}
  \label{r9}
  For $m,m',n\ge 0$, there are group homomorphisms
  \begin{gather*}
    \pi =\pi _{m,m',n}\zzzcolon \modH (m+m',n)\rightarrow \modH (m,n),\quad \pi (f)=f(1_m\otimes \eta _{m'}),\\
    \pi '=\pi '_{m,m',n}\zzzcolon \modH (m+m',n)\rightarrow \modH (m',n),\quad \pi (f)=f(\eta _m\otimes 1_{m'}).
  \end{gather*}
\end{lemma}

\begin{proof}
  Clearly, we have $\pi (\eta _n\epsilon _{m+m'})=\eta _n\epsilon _m$.  For
  $f,g\in \modH (m+m',n)$, we have
  \begin{gather*}
    \begin{split}
      \pi (f*g)
      &=\mu _n(f\otimes g)\Delta _{m+m'}(1_m\otimes \eta _{m'})\\
      &=\mu _n(f\otimes g)(1_m\otimes P_{m,m'}\otimes 1_{m'})(\Delta _m\otimes \Delta _{m'})(1_m\otimes \eta _{m'})\\
      &=\mu _n(f\otimes g)(1_m\otimes P_{m,m'}\otimes 1_{m'})(\Delta _m\otimes \eta _{m'}\otimes \eta _{m'})\\
      &=\mu _n(f\otimes g)(1_m\otimes \eta _{m'}\otimes 1_m\otimes \eta _{m'})\Delta _m\\
      &=\mu _n(f(1_m\otimes \eta _{m'})\otimes g(1_m\otimes \eta _{m'}))\\
      &=\pi (f)*\pi (g).
    \end{split}
  \end{gather*}
  Hence $\pi $ is a homomorphism.  Similarly, we can prove $\pi '$ is a
  homomorphism.
\end{proof}

\begin{proposition}[well known]
  \label{r10}
 The category $\modH $ admits finite coproducts.  More precisely, the
 object $0$ is the initial object, and for $m,m'\ge 0$, $m+m'$ is a
 coproduct of $m$ and $m'$, with a coproduct diagram
  \begin{gather}
    \label{e25}
    \vcenter{\xymatrix@!C{
    m\ar[r]^(.4){1_m\otimes \eta _{m'}}
    &
    m+m'
    &
    m'.\ar[l]_(.4){\eta _m\otimes 1_{m'}}
  }}
  \end{gather}
\end{proposition}

\no{
\begin{remark}
  It is easy to see that $0$ is terminal, too.  However, $\modH $ does not
  admit products.
\end{remark}
}

\begin{proof}
  That $0$ is initial in $\modH $ follows easily from Lemma
  \ref{factorization}.

  To prove that \eqref{e25} is a coproduct diagram, it suffices to
  prove that for every $n\ge 0$ the homomorphism
  \begin{gather}
    \label{e31}
    \vee\zzzcolon \modH (m,n)\cv\times \modH (m',n)\cv\rightarrow \modH (m+m',n)\cv
  \end{gather}
  from \eqref{e29} and the homomorphism
  \begin{gather*}
    (\pi ,\pi ')\zzzcolon \modH (m+m',n)\cv\rightarrow \modH (m,n)\cv\times \modH (m',n)\cv
  \end{gather*}
  are inverse to each other.  We check this claim on the generators.

  First, we check $\vee(\pi ,\pi ')(y^\H_{i,j,0})=y^\H_{i,j,0}$ for
  $i\in \{1,\ldots,m+m'\}$, $j\in \{1,\ldots,n\}$.  If $1\le i\le m$, then
  \begin{gather*}
      \vee(\pi ,\pi ')(y^\H_{i,j,0})
      =\vee(y^\H_{i,j,0},\eta _n\epsilon _{m'})
      =y^\H_{i,j,0},
  \end{gather*}
  and if $m+1\le i\le m+n$, then
  \begin{gather*}
    \vee(\pi ,\pi ')(y^\H_{i,j,0})
    =\vee(\eta _n\epsilon _{m},y^\H_{i-m,j,0})
    =y^\H_{i,j,0}.
  \end{gather*}

  Conversely, let us check $(\pi ,\pi ')\vee(f,f')=(f,f')$ for the
  generators of $\modH (m,n)\cv\times \modH (m',n)\cv$:
  $(f,f')=(y^\H_{i,j,0},\eta _n\epsilon _{m'})$ with $i\in \{1,\ldots,m\}$,
  $j\in \{1,\ldots,n\}$ and $(f,f')=(\eta _n\epsilon _{m},y^\H_{i,j,0})$ for
  $i\in \{1,\ldots,m'\}$, $j\in \{1,\ldots,n\}$.  In the former case,
  \begin{gather*}
    (\pi ,\pi ')\vee(y^\H_{i,j,0},\eta _n\epsilon _{m'})
    =(\pi ,\pi ')(y^\H_{i,j,0})
    =(\pi (y^\H_{i,j,0}),\pi '(y^\H_{i,j,0}))
    =(y^\H_{i,j,0},\eta _n\epsilon _{m'}).
  \end{gather*}
  The latter case can be checked similarly.
\end{proof}

By the proof of Proposition \ref{r10}, it follows that we have a group
isomorphism
\begin{gather*}
  \modH (m,n)\cv\simeq \modH (1,n)\cv^m.
\end{gather*}

\section{Motivations from topology}
\label{sec:motiv-from-topol}

\subsection{Category of bouquets of circles}

Here we recall well-known facts about the cogroup structure on $S^1$
and the Hopf monoid $\F=\modZ $ in $\modF $.

Let $\Top_*/h$ denote the category of pointed topological spaces and
homotopy classes of pointed continuous maps.  Let $\modB $ denote the full
subcategory of $\Top_*/h$ with $\Ob(\modB )=\{\bigvee^nS^1\zzzvert n\ge 0\}$.  The
fundamental group gives an isomorphism of symmetric monoidal
categories
\begin{gather*}
  \pi _1\zzzcolon \modB \overset{\simeq}{\longrightarrow}\modF 
\end{gather*}
in a natural way.

It is well known that the circle $S^1$ has a structure of a {\em
cogroup}, see \cite{Arkowitz}.  This means that there are two maps
\begin{gather*}
  \Delta \zzzcolon S^1\rightarrow S^1\vee S^1,\quad
  \gamma \zzzcolon S^1\rightarrow S^1,
\end{gather*}
which, together with the natural maps
\begin{gather*}
  \mu \zzzcolon S^1\vee S^1\rightarrow S^1,\quad
  \eta \zzzcolon *\rightarrow S^1,\quad
  \epsilon \zzzcolon S^1\rightarrow *,
\end{gather*}
form a commutative Hopf monoid structure in $\Top_*/h$.  The maps $\Delta $
and $\gamma $ are the familiar ones that appear in the definition of the
fundamental group of topological spaces.  The fundamental group
functor maps this cogroup structure on $S^1$ to the commutative Hopf
monoid structure on $\pi _1(S^1)=\modZ =\F$ in $\modF $.

\subsection{Category of handlebody embeddings}
\label{sec:categ-handl-embedd}
Here we discuss the {\em category of handlebody embeddings}, denoted
by $\cH$, which appeared in \cite{H:nqfg}, and will be studied in
detail in \cite{H:b1}.  The opposite $\cH^\op$ is isomorphic to the
``category of bottom tangles in handlebodies'', $\cB$, in
\cite{H:bottom}, and to the ``category of special Lagrangian
cobordisms'' in \cite{CHM}.

Define $\cH$ as follows.  The objects are nonnegative integers.  The
morphisms from $m$ to $n$ are the isotopy classes of embeddings of a
genus $m$ handlebody $V_m$ into a genus $n$ handlebody $V_n$.  Here,
$V_m$ is the ($3$-dimensional) handlebody of genus $m$ obtained by
attaching $m$ $1$-handles on the top of a cube.  By an {\em embedding}
of $V_m$ into $V_n$ we mean an embedding which fixes the bottom face
of the cube.  The isotopy classes are taken through such embeddings.

It is observed in \cite{H:bottom} that the category $\cB$ has a
structure of a braided monoidal category, and there is a Hopf monoid
in~$\cB$.  (In fact, $\cB$ may be regarded as a subcategory of the
category of cobordisms of surfaces with boundaries parameterized by
$S^1$, introduced by Crane and Yetter \cite{Crane-Yetter} and Kerler
\cite{Kerler}.)  Thus $\cH\simeq\cB^\op$ is a braided monoidal
category and admits a Hopf monoid in it.

There is a braided monoidal functor
\begin{gather*}
  \pi _1^{\cH}\zzzcolon \cH\longrightarrow \modF ,
\end{gather*}
mapping $n\in \Ob(\cH)$ to $\pi _1(V_n)\simeq F_n$, such that for
$[f]\zzzcolon m\rightarrow n$ with $f\zzzcolon V_m\hookrightarrow V_n$, we have
\begin{gather*}
  \pi _1^{\cH}([f])=\pi _1(f)\zzzcolon F_m\rightarrow F_n,
\end{gather*}
where we identify $F_m$ with $\pi _1(V_m)$ for $m\ge 0$.
For each $m,n\ge 0$, the map
\begin{gather*}
  \pi _1^{\cH}\zzzcolon \cH(m,n)\rightarrow \modF (m,n)
\end{gather*}
is surjective, and not bijective (if $m\neq0$).  Note that two
morphisms $[f],[g]\zzzcolon m\rightarrow n$ in $\cH$ satisfy
$\pi _1^{\cH}([f])=\pi _1^{\cH}([g])$ if and only if the representative
embeddings $f,g\zzzcolon V_m\hookrightarrow V_n$ are homotopic fixing the
bottom face of the cube.  Note also that the quotient category
$\cH/(\text{homotopy})$ is naturally isomorphic to the category of
bouquets of circles, $\modB $, since $V_m$ is homotopy equivalent to
$\bigvee^mS^1$.  One may regard $\cH$ as a refinement of
$\cH/(\text{homotopy})\simeq\modB \simeq\modF $, obtained by replacing
homotopy with isotopy.

In \cite{H:b1}, we will give a presentation of $\cH\simeq\cB^\op$ as a
braided monoidal category, which may be regarded as a refinement of
the presentation of $\modF $.  In $\cH$ there is a Hopf monoid, which is
mapped by $\pi _1^{\cH}$ into the Hopf monoid $\F$ in $\modF $.

\subsection{Category of chord diagrams in handlebodies}
\label{sec:categ-chord-diagr}
 Here we discuss the category $\bA$ of chord diagrams in handlebodies,
which will appear in a joint work with Massuyeau \cite{HM}.

As in the previous subsection, let $\cB\simeq\cH^\op$ denote the
category of bottom tangles in handlebodies \cite{H:bottom}.  In
\cite{HM}, using the Kontsevich integral, we will construct a functor
\begin{gather*}
  Z\zzzcolon \cB\longrightarrow\bA,
\end{gather*}
where $\bA$ is the ``category of chord diagrams in handlebodies'',
which will be defined in \cite{HM}.  Here we mention only the
following, ignoring some technical details such as completions and
non-associative monoidal structures, which usually appears in the
category-theoretic study of the Kontsevich integral and its variants.
The objects in $\bA$ are nonnegative integers.  The Hom space
$\bA(m,n)$ is the $\modQ $-vector space spanned by ``chord diagrams'' on
$n$ strands based at the bottom edge in a $2$-dimensional handlebody
obtained by attaching $m$ $1$-handles on the top of a square.  The
vector space $\bA(m,n)$ is graded by the number of chords.  Then the
degree $0$ part $\bA_0(m,n)$ of $\bA(m,n)$ is isomorphic to
$\modQ \modH ^\op(m,n)$, the $\modQ $-vector space with basis $\modH ^\op(m,n)$.
Thus, we may identify the linear category $\modQ \modH ^\op$ as the degree $0$
part $\bA_0$ of the graded linear category $\bA$.

The functor $Z$ is related to the functor $\pi _1^{\cH}\zzzcolon \cH\rightarrow \modF $ by the
following commutative diagram
\begin{gather}
  \label{e6}
  \vcenter{\xymatrix{
    \cB\ar[d]_{Z}\ar[r]_\simeq&
    \cH^\op\ar@{->>}[r]^{(\pi _1^{\cH})^\op}&
    \modF ^\op\ar[r]_\simeq^{(T^\op)^{-1}}&
    \modH ^\op\ar@{^{(}->}[d]^{\mathrm{incl}}
    \\
    \bA\ar@{->>}[rr]_{\mathrm{proj}}&&
    \bA_0\ar[r]_{\simeq}&
    \modQ \modH ^\op
  }}
\end{gather}

The category $\bA$ is a $\modQ $-linear symmetric monoidal category,
admitting a cocommutative Hopf monoid in the degree $0$ part $\bA_0$.
In \cite{HM}, we will give a presentation of $\bA$ as a $\modQ $-linear
symmetric monoidal category.

\end{document}